\documentclass[12pt]{article}
\usepackage{graphicx}%
\usepackage{multirow}%
\usepackage{amsmath,amssymb,amsfonts}%
\usepackage{amsthm}%
\usepackage{mathrsfs}%
\usepackage[title]{appendix}%
\usepackage{xcolor}%
\usepackage{textcomp}%
\usepackage{manyfoot}%
\usepackage{booktabs}%
\usepackage{algorithm}%
\usepackage{algorithmicx}%
\usepackage{algpseudocode}%
\usepackage{listings}%
\usepackage[numbers]{natbib}
\usepackage[colorlinks,citecolor=blue,pdfstartview=FitV,,urlcolor=blue,filecolor=blue]{hyperref}

\numberwithin{equation}{section}

\theoremstyle{plain}
\newtheorem{theorem}{Theorem}[section] 
\newtheorem{proposition}{Proposition}[section]
\newtheorem{definition}{Definition}[section]%
\newtheorem{corollary}{Corollary}[section]%
\newtheorem{assumption}{Assumption}[section]%
\newtheorem{lemma}{Lemma}[section]%
\newcommand{\chuhao}{\fontsize{19pt}{\baselineskip}\selectfont}

\date{}

\begin{document}
\title{\bf\color{black} \chuhao {On the Precision of the Spectral Profile Bound for the Mixing Time of Continuous State Markov Chains}}

\author{
	Elnaz Karimian Sichani \thanks{Corresponding author. Department of Mathematics and Statistics, University of Ottawa, Ottawa, ON, Canada. Email: \url{ekari037@uottawa.ca}}
	\and Aaron Smith \thanks{Department of Mathematics and Statistics, University of Ottawa, Ottawa, ON, Canada. Email: \url{asmi28@uottawa.ca}}
}

\maketitle	

	\begin{abstract}
		We investigate the \textit{sharpness} of the spectral profile bound presented by Goel et al. \cite{goel2006mixing} and Chen et al. \cite{chen2020fast} on the $L^{2}$ mixing time of Markov chains on continuous state spaces. We show that the bound provided by Chen et al. \cite{chen2020fast} is sharp up to a factor of $\log\log$ of the initial density. This result extends the findings of Kozma \cite{kozma2007precision}, which showed the analogous result for the original spectral profile bound of Goel et al. \cite{goel2006mixing} for Markov chains on finite state spaces. Kozma \cite{kozma2007precision} shows that the spectral profile bound is sharp up to a multiplicative factor of $\log(\log(\pi_{min}))$, where $\pi_{\min}$ is the smallest value of the probability mass function of the stationary distribution. We discuss the application of our primary finding to the comparison of Markov chains. Our main result can be used as a comparison bound, indicating that it is possible to compare chains even when only \textit{non-spectral} bounds exist for a known chain.

\par
{\small \vspace{9pt} \noindent\textbf{{ Keywords}:}{ Markov chains, Mixing time, Precision, Sensitivity, Continuous state space}}\\
\par
{\small \vspace{-3pt} \noindent\textbf{MSC 2010 subject classifications}: 60J05}\\
\vskip 0.2cm
 
	\end{abstract}

\section{Introduction}

Markov chains are popular both as models and as algorithms, in both contexts the \textit{mixing time} is often of great importance \cite{diaconis2009markov, diaconis2011mathematics}. Many of the most popular approaches to bounding the mixing time of a Markov chain are closely related to the spectrum of the underlying transition matrix \cite{levin2017markov}, both because they are easy to use and they are stable under natural changes to the underlying Markov chain \cite{kozma2007precision, levin2017markov}. This leads to the natural question: is it possible to find bounds that are both (almost) sharp and recognizably ``spectral'' or ``geometric'' in nature?

Many papers have investigated problems of a similar character, with the details depending on which notion of mixing must be approximated and what information the ``spectral'' or ``geometric'' bounds are allowed to use; some examples include  \cite{jh,kozma2007precision, hermon_peres_2018, hermon_kozma_2023, addario2018mixing}. The most similar to ours is \cite{kozma2007precision}, which shows that the \textit{spectral profile} of \cite{goel2006mixing} provides nearly-sharp bounds on the $L^{\infty}$ mixing time. More precisely, it shows that the spectral profile bound is sharp up to a multiplicative factor of $\log(\log(\pi_{\min}))$, where $\pi_{\min}$ is the smallest value of the PMF of the stationary distribution.

The main result of this paper, Theorem \ref{t2}, is an extension of Kozma's result \cite{kozma2007precision} from the original discrete-space spectral profile bound of \cite{goel2006mixing} to the continuous-space setting of \cite{chen2020fast}.
Of course the $L^{2}$ mixing time of a continuous-state Markov chain with non-zero holding probability is infinite, and the result of \cite{kozma2007precision} would be vacuous in a continuous-space setting.
We prove that Kozma's result \cite{kozma2007precision} is sharp with respect to a closely related metric (see Definition \ref{d2}), informally showing that this non-zero holding probability is the only thing that goes wrong. Our result replaces $\pi_{\min}$ with a ``warm-start'' constant that is equal to $\pi_{\min}$ in the discrete-space setting. This is a popular replacement in extending geometric bounds from discrete to continuous spaces \citep{vempala2005geometric}. 

As an auxiliary result, we note in Corollary \ref{application} how our main result can be used to get comparison bounds that are similar to the popular comparison bounds exposited in e.g. \cite{levin2017markov}. The main difference is that typical comparison bounds as in \cite{levin2017markov} require you to bound the spectrum of a ``nice'' kernel $K$ and allow you to use this to bound the spectrum of a ``difficult'' kernel $K'$;  Corollary \ref{application} points out that it is possible to compare chains even when you have only \textit{non-spectral} bounds on $K$.

\subsection{Related Work}

We mention some of the most closely-related work on how small changes in graph properties can affect the mixing times of associated random walks.

Hermon \cite{jh} explored the impact of bounded perturbations on the $L^{\infty}$ mixing times of simple random walks on graphs with uniformly bounded degrees. The study reveals that such perturbations can cause the mixing time to increase by a factor of $\Theta(\log \log n)$, where $n$ is the size of the graph. This highlights the sensitivity of the $L^{\infty}$ mixing time to changes in edge weights.

Hermon and Kozma \cite{hermon_kozma_2023}  investigated the robustness of the total variation mixing time in vertex-transitive graphs, particularly Cayley graphs, under small perturbations. Their findings indicate that for non-transitive graphs, the mixing time can vary significantly based on the starting point, especially after increasing certain edge weights.

Finally, Hermon and Peres \cite{hermon_peres_2018} examined the sensitivity of the total variation mixing times and the presence of a cutoff. The study shows that the total variation mixing time is not invariant under quasi-isometry, even for Cayley graphs, and can be substantially altered by bounded perturbations of edge weights or metric changes.

\subsection{Paper Guide}
In section \ref{sec2}, we establish our notation.  Section \ref{sec3} presents our main findings and includes an application. The final section \ref{sec5} is devoted to proofs of the theorems and lemmas presented in the paper.

\section{Notations and Basic Definitions}\label{sec2}

We set definitions used throughout the paper. Fix a Polish space $\mathcal{X}$ with associated Borel-$\sigma$ algebra $\mathcal{B}(\mathcal{X})$. Fix also a reversible transition probability kernel $P$ on $(\mathcal{X},\mathcal{B}(\mathcal{X}))$ with unique stationary probability measure $\pi$. Recall that a probability measure $\pi$ on $\mathcal{X}$ is called \textit{stationary} for the transition kernel $P$ if
$$
\int_{x \in \mathcal{X}}  P(x, A)\pi(dx)=\pi(A), \quad   \forall A \in \mathcal{B}(\mathcal{X})
$$
and $P$ is called \textit{reversible} if
$$
\int_{x \in A} \int_{y \in B}  \pi(dx)  P(x, d y) =\int_{x \in A} \int_{y \in B} \pi(dy) P(y, d x)  \quad  \forall A, B \in \mathcal{B}(\mathcal{X}).
$$ 

Define the $k$-step transition kernel $P^k$ recursively by \[P^{k+1}(x, A)=\int_{z \in \mathcal{X}} P(z, A) P^k(x, d z).\] For starting measure $\mu$, we define the density $h_{\mu, k}(x):= \frac{d(\mu P^{k})}{d\pi}(x)$ and consider only $\mu$ for which it exists.

We use the following definitions to measure the mixing rate of $P$:

\begin{definition}[$L^{p}$-distance]
	The $L^{p}$-distance between any two distributions $Q$ and $Q'$ is defined as: 
	$$		d_{p}(Q, Q')=\left(\int_{\mathcal{X}}\left\lvert \frac{dQ}{dQ'}(x)-1\right\rvert^{p} Q'(dx) \right)^{\frac{1}{p}},
	$$
	for $1\leq p<\infty$, and
	$$
	d_{\infty}(Q, Q')=\operatorname{esssup}_x\left|\frac{d Q}{d Q'}(x)-1\right|.
	$$
\end{definition}

\begin{definition}[ $L^{p}$ mixing time from an initial distribution ]
	The $L^{p}$ mixing time of a transition kernel $P$ from an initial distribution $\mu$ and with respect to the stationary distribution $\pi$ is:	
	$$
	\tau_{p}(\epsilon ; \mu, P)=\inf \left\{k \in \mathbb{N} \mid \, d_{p}(\mu P^{k}, \pi)\leq \epsilon\right\},
	$$
	where $\epsilon > 0$ is an error tolerance. 
\end{definition}

Denote by $\delta_{x}$ the usual Dirac-delta function at a point $x$. The  mixing time from an initial state $x$ is defined as follows:
$$
\tau_{p}(\epsilon; P)=\sup_{x}\tau_{p}(\epsilon ; \delta_{x}, P)=\inf \left\{k \in \mathbb{N} \mid \, \sup_{x \in \mathcal{X}}d_{p}(\delta_{x} P^{k}, \pi)\leq \epsilon\right\}.
$$

Note that, if $P(x,\{x\}) > 0$ for any point $x \in \mathcal{X}$ but $\pi$ has no atoms, then $\tau_{2}(\epsilon; P)$ is infinite. We are primarily interested in such chains. By a small abuse of standard notation, we define:

\begin{definition}[Exactly Half-Lazy]\label{as1} 
	A Markov chain with transition kernel $P$ is called \textit{exactly half-lazy} if it can be written in the form:
	\begin{equation}\label{eq4}
		P(x, A)=\frac{1}{2}\tilde{P}(x, A)+\frac{1}{2} \delta_{x}(A),
	\end{equation}
	for some transition kernel $\tilde{P}$ whose distributions $\tilde{P}(x,\cdot)$ all have densities $\tilde{p}(x,y)$ with respect to $\pi$. 
\end{definition}

For the remainder of the paper, we assume that $P$ is exactly half lazy and define $\tilde{P}$ as in Equation \eqref{eq4}.

\begin{definition}[ The $\delta$-approximate $L^{2}$ mixing time from a single starting point]\label{d2}   Consider an exactly half-lazy Markov chain $P$. For $x\in \mathcal{X}$, define
	
	\begin{equation}\label{e17}
		d_{2,\delta}(x, P,k)=
		\left\{\begin{array}{ll}
			d_{2}(\delta_{x} \tilde{P}\,P^{\lceil \frac{1}{2} k\rceil}, \pi)+\delta   & \quad P^{\lceil \frac{1}{2} k\rceil}(x,\{x\}) \leq \delta \\
			\infty & \quad 0 \leq \delta < P^{\lceil \frac{1}{2} k\rceil}(x,\{x\}).
		\end{array}\right.
	\end{equation}
	We then define the $\delta$-approximate $L^{2}$ mixing time of $P$ with respect to its stationary distribution $\pi$ from a single starting point $ x \in \mathcal{X}$,  $\tau_{2,\delta}(\epsilon; P)$, as        
	\begin{equation}\label{e18}
		\begin{aligned}
			\tau_{2,\delta} (\epsilon; P)
			&=	\inf \left\lbrace k \in \mathbb{N} \mid \, \sup_{x \in \mathcal{X}} d_{2,\delta}(x, P, k) \leq \epsilon \right\rbrace.
		\end{aligned}
	\end{equation}
\end{definition}

We say that a Markov chain with state space $\mathcal{X}$ and stationary distribution $\pi$ has a $\beta$-warm start if its initial distribution is somewhat spread out, as defined below:

\begin{definition}[Warm start]  A distribution $\mu$ is a $\beta$-warm start for a Markov chain  with stationary distribution $\pi$ if: 
	$$
	\sup _{A \in \mathcal{B}(\mathcal{X})} \frac{\mu(A)}{\pi(A)} \leq \beta.
	$$
\end{definition} 

The initial distribution $\mu$ is often referred to as $\beta$-warm start or a warm start with constant $\beta$. 

The mixing time of a Markov chain is closely related to its Dirichlet form:
	\begin{definition}[Dirichlet form] For a $\pi$-reversible transition kernel $P$, define the Dirichlet form on $L^2(\pi)$ by
	$$
	\mathcal{E}(f,f)=\frac{1}{2}\int_{\mathcal{X}}\int_{\mathcal{X}}\left(f(x)-f(y)\right)^{2} \pi(d x) P(x, d y)=\langle(I-P) f, f\rangle_{\pi},
	$$
	where $\langle\cdot, \cdot\rangle_{\pi}$ is the inner product on ${L}^{2}(\pi).$
	
\end{definition} 

We recall the generalization of spectral gap introduced in \cite[Definition 1.4]{goel2006mixing}:
\begin{definition}[Spectral gap]\label{d5}
	For  non-empty subset $A \subset \mathcal{X}$, the  spectral gap for the set $A$ is defined as
	\begin{equation}\label{e111}
		\gamma(A)=\inf _{f \in c_{0}^{+}( A)} \frac{\mathcal{E}(f, f)}{\operatorname{Var}_{\pi}(f)},
	\end{equation}
	where  $c_{0}^{+}(A)=\{f \in L^{2}(\pi): \operatorname{supp}(f) \subset  A, f \geq 0, f \neq \text { constant }\}$.
\end{definition}
This leads to the spectral profile, introduced in \cite{goel2006mixing}:
\begin{definition}[Spectral profile]\label{restspec}  For  $v \in [0, \infty)$, the spectral profile is defined
	\begin{equation}\label{e110}
		\Gamma(v)=\inf _{A \,:\,\pi(A) \in[0, v]} \gamma(A).
	\end{equation}
\end{definition}
The following is a specialization of Lemma 11 of \cite{chen2020fast} to our setting:

\begin{lemma}[Lemma 11, \cite{chen2020fast}]\label{l1}
	Consider a reversible, irreducible and exactly half-lazy continuous state Markov chain $P$ with the stationary distribution $\pi$. Given  a $\beta$-warm start $\mu$, and an error tolerance $\epsilon \in (0,1)$, the $L^{2}$ mixing time, $\tau_{2}\left(\epsilon ; \mu, P\right)$, is bounded as:
	\begin{equation}\label{e1}
		\tau_{2}\left(\epsilon ; \mu, P\right) \leq \int_{4 / \beta}^{8 / \epsilon^{2}} \frac{2\; d v}{v \Gamma(v)}.
	\end{equation}
\end{lemma}

\section{Precision of the Spectral Profile Bound in Continuous Setting}\label{sec3}

In this section, we extend Theorem 1 of \cite{kozma2007precision} to the continuous state setting.  We first establish an analogue to Lemma 3.1 of \cite{goel2006mixing} in the continuous state space. This result is then applied to our Lemma \ref{l2}, which is similar to Lemma \ref{l1} but from a single starting point. All proofs are deferred to Section \ref{sec5}.

\begin{lemma}\label{l2}
	Consider a reversible, irreducible, and exactly half-lazy continuous state Markov chain $P$ with the stationary distribution $\pi$. Assume that $\delta_{x} \tilde{P}$ is a $\beta$-warm start for every $x \in \mathcal{X}.$  Given error tolerances $\epsilon \in (0,1)$ and $ \delta \in [0, \epsilon),$ then
	\begin{eqnarray}
		\tau_{2,\delta}(\epsilon; P)  
		&\leq & \max \left(2\,\left(\int_{4 / \beta}^{8 / (\epsilon-\delta)^{2}} \frac{2\; d v}{v \Gamma(v)} \right), \; \frac{2\log(\frac{1}{\delta})}{\log({2})} +1 \right).
		\label{e7}
	\end{eqnarray}
\end{lemma}

Throughout this paper,  the term ``$\rho$'' is used to refer to the right-hand side of inequality \eqref{e7}, and the symbol $\log$ is used to represent the natural logarithm. 

For $S \in \mathcal{B}(\mathcal{X})$ with $\pi(S) >0$, define a sub-stochastic kernel
\begin{equation*}
	\tilde{P}_{S}(x, B)= \tilde{P}(x,B \cap S), \quad \forall  B \in \mathcal{B}(\mathcal{X}).
\end{equation*}
Note that the kernels $\tilde{P}$ and $\tilde{P}_{S}$ are reversible with respect to $\pi$ and  ${\pi}_{\mid{S}}(B) \equiv \frac{{\pi}(B \cap S)}{{\pi}(S)}$, respectively. We make the following assumption on $\tilde{P}_{S}$:

\begin{assumption}\label{as2}
	For $S \in \mathcal{B}(\mathcal{X})$ with $\pi(S) >0$, we assume that the sub-stochastic kernel $\tilde{P}_{S}$ on $L^{2}({\pi}_{\mid{S}})$ is Hilbert-Schmidt.
	Assume that $L^{2}({\pi}_{\mid{S}})$ has an orthonormal basis of eigenfunctions $\left\{f_{i}\right\}_{i \geq 0}$ of  $\tilde{P}_{S}$, with real eigenvalues $\left\{\tilde{\beta}_{i}\right\}_{i \geq 0}$ satisfy $f_{0} \equiv 1$,  $0\leq \tilde{\beta}_{i} < 1$, $\tilde{\beta}_{i}\downarrow 0$ so that
	\begin{equation}\label{ascon}
		\int_{\mathcal{X}} \frac{d(\delta_x\tilde{P}_{S})}{d{\pi}_{\mid{S}}}(y) {f}_i\left(y\right) {\pi}_{\mid{S}}\left(d y\right)=\tilde{\beta}_i {f}_i(x),  \quad \forall  x \in \mathcal{X}.
	\end{equation}
\end{assumption}

In order to prove Theorem \ref{t2}, we then propose the following lemma, which is similar to Lemma 3.1 in \cite{goel2006mixing}. 

\begin{lemma}\label{l3} 
	Let $S \in \mathcal{B}(\mathcal{X})$ with $\pi(S) >0$ and  $k \in \mathbb{N} $. Let $\tilde{P}_{S}$ satisfy Assumption \ref{as2}, then
	\begin{equation}\label{rflo}
		\sup_{x \in \mathcal{X}}{h}_{\delta_{x} \tilde{P}, \lceil k/2 \rceil}(x) 
		\geq
		\frac{(1-\tilde{\gamma}(S))^{2k}}{\pi(S)},
	\end{equation}
	where $\tilde{\gamma}(S)$ is the spectral gap of $\tilde{P}$ for the set $S$.
\end{lemma}

Following Lemmas \ref{l2} and \ref{l3}, we can extend Theorem 1 in \cite{kozma2007precision} to the continuous state space context as Theorem \ref{t2} in this study. This represents the precision of the spectral profile bound for $ \tau_{2,\frac{1}{8}}(\frac{1}{4}; P)$.

\begin{theorem}\label{t2}
	Consider a reversible, irreducible, and exactly half-lazy Markov chain $P$ with the stationary distribution $\pi$. Assume that $\delta_{x} \tilde{P}$ is a $\beta$-warm start for every $x \in \mathcal{X}$ and $\tilde{P}$ satisfies Assumption \ref{as2}.
	Then there exists a universal constant $C$ such that
	\begin{equation}\label{mainresult}
				\tau_{2,\frac{1}{8}}(\frac{1}{4}; P) \leq \rho \leq  C(\log \lceil\log_{2} (\beta)\rceil+1) \tau_{2,\frac{1}{8}}(\frac{1}{4}; P) +108\, \lceil\log_{2} (\beta)\rceil + 7.
	\end{equation}
\end{theorem}

The proof of Theorem \ref{t2} is quite similar to the proof of Theorem 1 in \cite{kozma2007precision}, with the substitution of Lemma 3.1 in \cite{goel2006mixing} by our Lemma \ref{l3}.

Theorem \ref{t2} can be used to compare Markov chains in the strong $L^{2}$ metric:

\begin{corollary}\label{application}
	Consider two Markov chains, $\tilde{K}$ and $\tilde{K}'$, which are reversible, irreducible, and satisfy Assumption \ref{as2}. Let $K$ and $K'$ be their respective exactly half-lazy versions, with respect to the stationary distribution $\pi$, where $\delta_{x} \tilde{K}$ is a $\beta$-warm start for every $x$. Denote the spectral profiles of $K'$ and $K$ as $\Gamma_{K'}$ and $\Gamma_{K}$, respectively. If there exists $ 0<C_{1}<\infty$ so that      $$\Gamma_{K'}(v) \geq \frac{1}{C_{1}} \Gamma_{K}(v), \; \forall v \in [0,\infty),$$
	then there exists a universal constant $C$ such that
	$$
	\tau_{2,\frac{1}{8}}(\frac{1}{4} ; K') \leq 
	C_{1}\, C(\log \lceil\log_{2} (\beta)\rceil+1) \left(\tau_{2}(\frac{1}{8} ; \delta_{x}\tilde{K}, K) +\frac{1}{8}\right) + 108\, \lceil\log_{2}(\beta)\rceil+ 7.\\
	$$
\end{corollary}

\section{Proofs}\label{sec5}

\begin{proof}[\textbf{Proof of Lemma ~{\upshape\ref{l2}}}]\label{pl1}
	Conditioning on the first time a Markov chain moves,
	
	\begin{eqnarray}
		\delta_{x}P^{s}&=&\delta_{x}PP^{s-1}\nonumber\\
		&=&\delta_{x}\left[ 2^{-1}\tilde{P}+2^{-1}  I\right] P^{s-1}\nonumber\\
		&=&2^{-1}\delta_{x}\tilde{P}P^{s-1}+2^{-1} \delta_{x}P^{s-1}\nonumber\\
		&=&2^{-1}\delta_{x}\tilde{P}P^{s-1}+2^{-1}\delta_{x}PP^{s-2}\nonumber\\
		&\vdots&\nonumber
	\end{eqnarray} 
	By iteratively expanding $P = 2^{-1}\tilde{P}+2^{-1} \delta_{x}$ and collecting terms, we obtain
	\begin{equation}\label{e19}
		\delta_{x} P^{s} =  \sum_{n=1}^{s}2^{-n}\delta_{x}\tilde{P}P^{s-n} + 2^{-s} \delta_{x}.
	\end{equation} 
	Since
	\begin{eqnarray} \label{e666}
		\tau_{2,\delta}(\epsilon;P)
		&=&	\inf \left\lbrace k \in \mathbb{N} \mid \, \sup_{x \in \mathcal{X}}d_{2,\delta}(x, P,k) \leq \epsilon \right\rbrace\nonumber \\
		&= &\inf \left\{k \in \mathbb{N} \mid \, \sup_{x \in \mathcal{X}}d_{2}(\delta_{x} \tilde{P}\,P^{\lceil \frac{1}{2} k\rceil}, \pi)\leq \epsilon-\delta, \; 2^{-\lceil \frac{1}{2} k\rceil}\leq\delta \right\},
	\end{eqnarray} 
	this implies
	
	\begin{equation} \label{e9}
		\tau_{2,\delta}(\epsilon; P) \leq  \max \left(2\,\sup_{x \in \mathcal{X}}\tau_{2}(\epsilon-\delta; \delta_{x}\tilde{P}, P), \; \frac{2\log(\frac{1}{\delta})}{\log({2})} +1 \right).
	\end{equation}
	Applying Lemma \ref{l1} to the term ``$\sup_{x \in \mathcal{X}}\tau_{2}(\epsilon-\delta; \delta_{x}\tilde{P}, P)$'' in inequality \eqref{e9} completes the proof.
	
\end{proof}

\begin{proof}[\textbf{Proof of Lemma ~{\upshape\ref{l3}}}]\label{pl3}
	By Assumption \ref{as2}:
	$$\frac{d(\delta_{x}\tilde{P}_{S}^{\ell})}{d{\pi}_{\mid{S}}}(y)=\sum_{i=0}^{\infty} \tilde{\beta}_{i}^{\ell}{f}_{i}(x){f}_{i}(y), \quad \forall x, y \in {\mathcal{X}} \quad \text{and} \quad \forall \, \ell \in \mathbb{N},
	$$
	where $\tilde{\beta}_i$ and ${f}_i$ are eigenvalues and  orthonormal eigenfunctions of $\tilde{P}_{S}$, respectively.
	Letting $y=x$, we then have
	
	\begin{equation}\label{tilp}
		\frac{d(\delta_{x}\tilde{P}_{S}^{\ell})}{d{\pi}_{\mid{S}}}(x)=\sum_{i=0}^{\infty} \tilde{\beta}_{i}^{\ell}{f}_{i}^{2}(x), \quad \forall x\in \mathcal{X}.
	\end{equation}
	
	For $s,j \in \mathbb{N}_{0}$, $j \leq s $, define $b(s,j) = 2^{-s} \binom{s}{j}$. We have:
	\begin{equation}\label{eqp}
		P^{s}=\sum_{j=0}^{s} b(s,j) \tilde{P}^j.
	\end{equation}
	
	Thus, we can write $\delta_x\tilde{P}P^{\lceil \frac{1}{2}k \rceil}$ as a polynomial in $\tilde{P}$ as follows:
	
	\begin{eqnarray}\label{polyQ}
		\delta_{x}\tilde{P}P^{\lceil \frac{1}{2}k \rceil} &=&
		\delta_{x}\tilde{P}\left(\sum_{j=0}^{\lceil \frac{1}{2}k \rceil} b(\lceil \frac{1}{2}k \rceil,j) \tilde{P}^j\right) \nonumber \\
		&=&\sum_{j=0}^{\lceil \frac{1}{2}k \rceil} b(\lceil \frac{1}{2}k 
		\rceil,j)  \delta_{x}\tilde{P}^{j+1}.     
	\end{eqnarray}
	Using Equation \eqref{tilp}, we can derive that
	$\forall x \in S$ and  $k \in \mathbb{N} $
	
	\begin{eqnarray}
		\pi(S) \frac{\ d(\delta_{x}\tilde{P}P^{\lceil \frac{1}{2}k \rceil})}{ d{\pi}}(x) 
		&\stackrel{(i)} {=}&
		\frac{\ d(\delta_{x}\tilde{P}P^{\lceil \frac{1}{2}k \rceil})}{ d{\pi}_{\mid{S}}}(x) \nonumber\\ 
		&\stackrel{(ii)}=&
		\frac{d(\sum_{j=0}^{\lceil \frac{1}{2}k \rceil} b(\lceil \frac{1}{2}k \rceil,j)  \delta_{x}\tilde{P}^{j+1})}{d{\pi}_{\mid{S}}}(x)   \nonumber\\ 
		&\geq&  
		\frac{d(\sum_{j=0}^{\lceil \frac{1}{2}k \rceil} b(\lceil \frac{1}{2}k \rceil,j)  \delta_{x}\tilde{P}_{S}^{j+1})}{d{\pi}_{\mid{S}}}(x)   \nonumber\\ 
		&=&
		\sum_{j=0}^{\lceil \frac{1}{2}k \rceil}b(\lceil \frac{1}{2}k \rceil,j)\frac{d(   \delta_{x}\tilde{P}_{S}^{j+1})}{d{\pi}_{\mid{S}}}(x)   \nonumber\\
		&\stackrel{(iii)}=& \sum_{j=0}^{\lceil \frac{1}{2}k \rceil}b(\lceil \frac{1}{2}k \rceil,j) \left(\sum_{i=0}^{\infty} \tilde{\beta}_{i}^{j+1}{f}_{i}^{2}(x) \right) \nonumber \\
		&\stackrel{(iv)} \geq &  
		\sum_{j=0}^{\lceil \frac{1}{2}k \rceil}b(\lceil \frac{1}{2}k \rceil,j) \left(\tilde{\beta}_{1}^{j+1}{f}_{1}^{2}(x)\right)
		\nonumber \\
		&\geq&
		{f}_{1}^{2}(x) \tilde{\beta}_{1}^{2k}.\nonumber
	\end{eqnarray}	
	Step (i) uses ${\pi}_{\mid{S}}(B)=\frac{{\pi}(B \cap S)}{{\pi}(S)}$, $\forall B \subset \mathcal{X}$. 
	Step (ii) uses Equation \eqref{polyQ}. Step (iii) applies Equation \eqref{tilp} and
	step (iv) follows from Assumption \ref{as2} that $0\leq \tilde\beta_i \leq 1$ and the fact that ${f}_i^{2}(x)\geq 0$.
	
	By taking the supremum over $x$  
	\begin{eqnarray}
		 \sup_{x \in \mathcal{X}}{h}_{\delta_{x} \tilde{P}, \lceil k/2 \rceil}(x)
		 &\geq&
		\sup_{x \in S}{h}_{\delta_{x} \tilde{P}, \lceil k/2 \rceil}(x) \nonumber \\
		&=&\sup_{x \in S}\frac{\ d(\delta_{x}\tilde{P}P^{\lceil \frac{1}{2}k \rceil})}{ d{\pi}}(x) \nonumber \\
		&\geq& \frac{\tilde{\beta}_{1}^{2k}\sup_{x \in S}{f}_{1}^{2}(x)}{\pi(S)} \nonumber \\
		&\stackrel{(i)} \geq& \frac{\tilde{\beta}_{1}^{2k}}{\pi(S)} \nonumber \\
		&=& \frac{(1-\tilde{\gamma}(S))^{2k}}{\pi(S)}, \nonumber
	\end{eqnarray}
	where $\tilde{\gamma}(S)$ is the spectral gap of $\tilde{P}$ for the set $S$. Inequality (i) follows from the fact that $\left\| {f}_{1} \right\|_{{\pi}_{\mid{S}}}^{2}=1$.

\end{proof}

We are now equipped to prove Theorem \ref{t2}.
\begin{proof}[\textbf{Proof of Theorem ~{\upshape\ref{t2}}}] \label{pt2}  Fix $C'>0$. Consider set $A_{s}$ so that $0< \pi\left( A_{s} \right) \leq 2^{-s}$ and 
	\begin{eqnarray}\label{e90}
		\gamma(A_{s})\leq \inf \left\{\gamma(S):
		\pi(S) \leq 2^{-s}\right\}+ C',
	\end{eqnarray}
	where $\gamma(A_{s})$ is the spectral gap of $P$ for the set $A_{s}$.
	By monotonicity of $\Gamma$ and change of variables:
	\begin{eqnarray}
		\int_{4 / \beta}^{512} \frac{2\,d v}{v \Gamma(v)}&=&(2 \log 2) \int_{-9}^{\log _2 \beta-2} \frac{d u}{\Gamma\left(2^{-u}\right)} \nonumber \\
		&=&(2 \log 2)\left[\int_0^{\log _2 \beta-2} \frac{d u}{\Gamma\left(2^{-u}\right)}+\frac{9}{\Gamma(1)}\right] \nonumber\\
		&\leq&  2 \log 2\sum_{s=1}^{\lceil\log_{2} (\beta)\rceil} \frac{1}{\Gamma\left(2^{-s}\right)} +\frac{18\log 2}{\Gamma(1)}.
	\end{eqnarray}
	Using Lemma 2.2 of \cite{goel2006mixing} and the definition of the spectral profile:
	\begin{equation}\label{gamma1}
		\Gamma(1/2)\leq 2\, \Gamma(1).
	\end{equation}
	
	Hence
	\begin{eqnarray}
		\int_{4 / \beta}^{512} \frac{2\,d v}{v \Gamma(v)}
		&\leq&  2 \log 2\sum_{s=1}^{\lceil\log_{2} (\beta)\rceil} \frac{1}{\Gamma\left(2^{-s}\right)} +\frac{36\log 2}{\Gamma(1/2)} \nonumber\\
		&\leq&  27\sum_{s=1}^{\lceil\log_{2} (\beta)\rceil} \frac{1}{\Gamma\left(2^{-s}\right)}.
	\end{eqnarray}

		We then have, for all $C' > 0$ sufficiently small,
	\begin{eqnarray}\label{e4}
		\rho &\leq & 54\sum_{s=1}^{\lceil\log_{2} (\beta)\rceil} \frac{1}{\Gamma\left(2^{-s}\right)} + \frac{2\;\log(8)}{\log(2)}+1 \nonumber \\
		&\stackrel{(i)}\leq& 54 \sum_{s=1}^{\lceil\log_{2} (\beta)\rceil} \frac{1}{\Gamma\left(2^{-s}\right)} + \frac{2\;\log(8)}{\log(2)}+1 \nonumber \\
		&\stackrel{(ii)}\leq& 54\sum_{s=1}^{\lceil\log_{2} (\beta)\rceil} \frac{1}{\gamma(A_{s})-C'} + 7 \nonumber \\
		&\stackrel{(iii)}= & 108 \sum_{s=1}^{\lceil\log_{2} (\beta)\rceil} \frac{1}{\tilde{\gamma}(A_{s})-2C'} + 7.
	\end{eqnarray}
   	Step (i) follows from equation \eqref{gamma1}. Step (ii) follows from \eqref{e90}. Step (iii) applies the exactly half-lazy Definition \ref{as1}. Let $C'\rightarrow 0$,
   	and applying Lemma \ref{l3} for the set $A_{s}$ we have for any $k \in \mathbb{N}$
   	
	\begin{eqnarray}\label{reslemf}
		\sup_{x \in \mathcal{X}}{h}_{\delta_{x} \tilde{P}, \lceil k/2 \rceil}(x) 
		&\geq&
		\frac{(1-\tilde{\gamma}(A_{s}))^{2k}}{\pi(A_{s})} \nonumber \\
		&=& \frac{\exp(2k\log(1-\tilde{\gamma}(A_s))}{\pi(A_{s})}.
	\end{eqnarray}
	Imitating the steps of \cite{kozma2007precision},  we can deduce the following:   
	
	\begin{eqnarray}\label{e26}
		\tau_{2,\frac{1}{8}}(\frac{1}{4}; P) 
		&\stackrel{(i)}{\geq}& c \frac{\log \left(\pi(A_{s})\right)}{\log(1-\tilde\gamma(A_{s}))}\nonumber \\
		& \geq& c \frac{\log{(2^{-s}})}{\log(1-\tilde\gamma(A_{s}))}\label{sp} \nonumber \\
		& \stackrel{(ii)} \geq & c \frac{s\log{2}}{\frac{1}{1-\tilde\gamma(A_{s})}-1},
	\end{eqnarray}
	where $c$ denotes an absolute positive constant. Step (i) follows from inequality \eqref{reslemf}, Equation \eqref{e9}, and  Proposition \ref{appro1} in Appendix \ref{appA} for the reversible chain $Q=\tilde{P}P^{\lceil k/2 \rceil}$. 
	Step (ii) applies $\log(\frac{1}{x}) \leq \frac{1}{x}-1, \forall x>0 $.
	
	Note that $\frac{1}{1-\tilde\gamma(A_{s})}-1>0$.
	Therefore, we obtain the following using inequality \eqref{e26}
	\begin{eqnarray}
		\tau_{2,\frac{1}{8}}(\frac{1}{4}; P)\times \left( \frac{1}{1-\tilde\gamma(A_{s})}-1\right) 
		& \geq & c \,\log{2}\, s .	 \nonumber
	\end{eqnarray}
	As a result, we get
	\begin{eqnarray}\label{e31}
		\frac{1}{\tilde\gamma(A_{s})}
		\leq c' \frac{ \tau_{2,\frac{1}{8}}(\frac{1}{4}; P)}{s} + 1,
	\end{eqnarray}
	where $c'$ denotes an absolute positive constant. Finally, by combining \eqref{e4} and \eqref{e31}, we have
	
	\begin{eqnarray*}
		\rho &\leq& 108 \sum_{s=1}^{\lceil\log_{2} (\beta)\rceil} \left(  c' \frac{ \tau_{2,\frac{1}{8}}(\frac{1}{4}; P)}{s} + 1 \right)+7\\ \nonumber 
		& \stackrel{(i)}\leq & 108 \, c'(\log \lceil\log_{2} (\beta)\rceil+1) \tau_{2,\frac{1}{8}}(\frac{1}{4}; P) +108\, \lceil\log_{2} (\beta)\rceil + 7 \\ \nonumber
		&=& C(\log \lceil\log_{2} (\beta)\rceil+1) \tau_{2,\frac{1}{8}}(\frac{1}{4}; P) +108\, \lceil\log_{2} (\beta)\rceil + 7,  \nonumber
	\end{eqnarray*}
	where $C$ denotes an absolute positive constant. Inequality (i) follows by Riemann approximation. This completes the proof of Theorem \ref{t2}.
	
\end{proof}

\begin{proof}[\textbf{Proof of Corollary ~{\upshape\ref{application}}}]\label{papplication}
	
	\begin{eqnarray*}
	\tau_{2,\frac{1}{8}}(\frac{1}{4} ; K') &\leq& \rho'\\
	&\stackrel {(i)}\leq& C_{1} \rho \\
	& \stackrel{(ii)} \leq& C_{1} C(\log \lceil\log_{2} (\beta)\rceil+1) (\tau_{2}(\frac{1}{8} ; \delta_{x}\tilde{K}, K)+\frac{1}{8})+ 108\, \lceil\log_{2}(\beta)\rceil+ 7.    
\end{eqnarray*}
Step (i) follows from $\Gamma_{K'}(v) \geq \frac{1}{C_{1}} \Gamma_{K}(v), \; \forall v \in [0,\infty)$. Step (ii) follows from the precision inequality in Theorem \ref{t2}. 
	
\end{proof}

\section*{Acknowledgements}
AS is supported by the NSERC Discovery Grant Program.

\appendix

\section{Appendix}\label{appA}
	\begin{proposition}\label{appro1}
		For a reversible Markov chain with transition probability distribution $Q$ with respect to the stationary distribution $\pi$:
		$$
		\sup_{x}h_{\delta_x, 2n}(x)-1 \leq
		\sup_{ x \in \mathcal{X}}d_{2}^{2}(\delta_xQ^n, \pi),
		$$
		where  $
		h_{\delta_x,2n}(y):=\frac{d\delta_xQ^{2n}}{d\pi}(y)
		$ and $d_{2}^{2}(\delta_xQ^{n}, \pi)= \left\|h_{\delta_x, n}-1\right\|_2^{2}.$
	\end{proposition}
	
	\begin{proof}
		This proof is inspired by inequality (2.2) in \cite{goel2006mixing}.
		Note that $h_{\delta_x,n}(y)=h_{\delta_y,n}(x)$ when $Q$ is reversible with respect to $\pi$. 
		We have
		
		$$
		h_{\delta_x, 2n}(y):=\frac{d\delta_{x}Q^{2 n}}{d\pi}(y)= \frac{\delta_{x}Q^{2n}(dy)}{\pi(dy)}.
		$$
		Using reversibility, we then obtain
		$$
		\begin{aligned}
			\left|h_{\delta_x, 2n}(y)-1\right| &= \left|\frac{\delta_{x}Q^{2n}(dy)-\pi(dy)}{\pi(dy)}\right| \\ &=\left|\frac{\int_{z \in \mathcal{X}}\left(\delta_{x}Q^n(dz)-\pi(dz)\right)\left(\delta_{z}Q^n(dy)-\pi(dy)\right)}{\pi(dy)}\right| \\
			&=\left|\int_{z \in \mathcal{X}}\left[h_{\delta_x, n}(z)-1\right]\left[h_{\delta_y, n}(z)-1\right] \pi(dz)\right| \\
			& \leq\left\|h_{\delta_x, n}-1\right\|_2\left\|h_{\delta_y, n}-1\right\|_2,
		\end{aligned}
		$$
		where the the last inequality follows from Cauchy-Schwarz.
		Let $x=y$ and take the supremum over $x$ yields
		$$
		\begin{aligned}
			\sup_{x}h_{\delta_x, 2n}(x)-1 & \leq \sup_{x}\left\|h_{\delta_x, n}-1\right\|_2\left\|h_{\delta_x, n}-1\right\|_2 \nonumber \\ & =
			\sup_{ x \in \mathcal{X}}d_{2}^{2}(\delta_xQ^n, \pi).
		\end{aligned}
		$$
	\end{proof}
\bibliographystyle{plain}
\bibliography{bibliography}

\begin{thebibliography}{10}

\bibitem{addario2018mixing}
Louigi Addario-Berry and Matthew~I Roberts.
\newblock Mixing time bounds via bottleneck sequences.
\newblock {\em Journal of Statistical Physics}, 173:845--871, 2018.

\bibitem{chen2020fast}
Yuansi Chen, Raaz Dwivedi, Martin~J Wainwright, and Bin Yu.
\newblock {Fast mixing of Metropolized Hamiltonian Monte Carlo: Benefits of
  multi-step gradients}.
\newblock {\em Journal of Machine Learning Research}, 21(92):1--71, 2020.

\bibitem{diaconis2009markov}
Persi Diaconis.
\newblock {The Markov chain Monte Carlo revolution}.
\newblock {\em Bulletin of the American Mathematical Society}, 46(2):179--205,
  2009.

\bibitem{diaconis2011mathematics}
Persi Diaconis.
\newblock The mathematics of mixing things up.
\newblock {\em Journal of Statistical Physics}, 144(3):445--458, 2011.

\bibitem{goel2006mixing}
Sharad Goel, Ravi Montenegro, Prasad Tetali, et~al.
\newblock {Mixing time bounds via the spectral profile}.
\newblock {\em Electronic Journal of Probability}, 11:1--26, 2006.

\bibitem{jh}
Jonathan Hermon.
\newblock {On sensitivity of uniform mixing times}.
\newblock {\em Annales de l'Institut Henri Poincaré, Probabilités et
  Statistiques}, 54(1):234 -- 248, 2018.

\bibitem{hermon_kozma_2023}
Jonathan Hermon and Gady Kozma.
\newblock Sensitivity of mixing times of cayley graphs.
\newblock {\em Canadian Journal of Mathematics}, page 1–32, 2023.

\bibitem{hermon_peres_2018}
Jonathan Hermon and Yuval Peres.
\newblock On sensitivity of mixing times and cutoff.
\newblock {\em Electronic Journal of Probability}, 23, Jan 2018.

\bibitem{kozma2007precision}
Gady Kozma.
\newblock On the precision of the spectral profile.
\newblock {\em Alea}, 3:321--329, 2007.

\bibitem{levin2017markov}
David~A Levin and Yuval Peres.
\newblock {\em Markov chains and mixing times}.
\newblock American Mathematical Soc., United States of America, 2017.

\bibitem{vempala2005geometric}
Santosh Vempala.
\newblock {Geometric random walks: a survey}.
\newblock {\em Combinatorial and computational geometry}, 52(573-612):2, 2005.

\end{thebibliography}
\end{document}